\documentclass[a4paper]{amsart}

\usepackage{eucal}
\usepackage{tikz} \usepackage{tikz-cd}
\usetikzlibrary{calc} \makeatletter
\define@key{meshkeys}{r}{\def\myr{#1}}
\define@key{meshkeys}{a0}{\def\myaa{#1}}
\define@key{meshkeys}{a1}{\def\myab{#1}}
\tikzdeclarecoordinatesystem{mesh}%
{%
  \setkeys{meshkeys}{#1}%
  \pgfpointadd{\pgfpointxy{\myr*(0+\myaa*cos(45-0*(90/1))+\myab*cos(45-1*(90/1)))}%
    {-\myr*(0+\myaa*sin(45-0*(90/1))+\myab*sin(45-1*(90/1)))}}%
}

\theoremstyle{plain} \newtheorem*{proposition}{Proposition}
\theoremstyle{remark} \newtheorem*{remark}{Remark}

\newcommand{\set}[1]{\left\{\,#1\,\right\}}
\newcommand{\setP}[2]{\set{#1\ \middle|\ #2}}
\DeclareSymbolFont{sfoperators}{OT1}{cmss}{m}{n}
\DeclareSymbolFontAlphabet{\mathsf}{sfoperators}
\makeatletter\def\operator@font{\mathgroup\symsfoperators}\makeatother

\DeclareMathOperator{\add}{add} \DeclareMathOperator{\Db}{D^b}
\DeclareMathOperator{\mmod}{mod}
\DeclareMathOperator{\Ext}{Ext}

\newcommand{\rep}[1]{%
  {%
    \small%
    \begin{matrix}%
      #1%
    \end{matrix}%
  }%
}

\begin{document}

\date{\today}
  
\title[The naive approach for constructing the derived category fails]{The naive approach for constructing the derived category of a
  $d$-abelian category fails}

\author[G.~Jasso]{Gustavo Jasso}
\address[Jasso]{Mathematisches Institut\\
  Universit\"at Bonn\\
  Endenicher Allee 60\\
  D-53115 Bonn\\
  GERMANY} \email{gjasso@math.uni-bonn.de}
\urladdr{https://gustavo.jasso.info}

\author[J.~K\"ulshammer]{Julian K\"ulshammer}
\address[K\"ulshammer]{Institut f\"ur Algebra und Zahlentheorie\\
  Universit\"at Stuttgart\\
  Pfaffenwaldring 57\\
  70569 Stuttgart\\
  GERMANY} \email{julian.kuelshammer@mathematik.uni-stuttgart.de}
\urladdr{http://www.iaz.uni-stuttgart.de/LstAlgZahl/Kuelshammer/}

\begin{abstract}
  Let $k$ be a field. In this short note we give an example of a
  $2$-abelian $k$-category, realized as a $2$-cluster-tilting subcategory
  of the category $\mmod A$ of finite dimensional (right) $A$-modules
  over a finite dimensional $k$-algebra $A$, for which the naive idea
  for constructing its ``bounded derived category'' as
  $2$-cluster-tilting subcategory of the bounded derived category of
  $\mmod A$ cannot work.
\end{abstract}
\subjclass[2010]{Primary: 16G70. Secondary: 16G20}


\maketitle

\section*{Introduction}

Let $d$ be a positive integer. Motivated by Iyama's higher
Auslander--Reiten theory \cite{iyama_higher-dimensional_2007}, the
class of $d$-abelian categories was introduced in
\cite{jasso_n-abelian_2016} as an abstract framework for investigating
the intrinsic properties of $d$-cluster-tilting subcategories of
abelian categories.  Note that $d$-cluster-tilting subcategories of
abelian categories are $d$-abelian categories, see
\cite[Thm. 3.16]{jasso_n-abelian_2016}.  If for a $d$-abelian category
the projective and injective objects coincide, that is it is a
\emph{Frobenius $d$-abelian category}, then its stable category is a
$(d+2)$-angulated category in the sense of
\cite{geiss_n-angulated_2013}.  It is then a natural problem to
construct a ``derived category'' of a given $d$-abelian category as
well, which is expected to be a $(d+2)$-angulated category.

There is the following evidence in support of this expectation. Given
a $d$-representation-finite algebra $A$ in the sense of
\cite{iyama_n-representation-finite_2011}, it is known that there is a
unique $d$-cluster-tilting subcategory $\mathcal{M}(A)$ of $\mmod A$,
the category of finite dimensional (right) $A$-modules, see
\cite[Prop. 1.3]{iyama_cluster_2011}. We remind the reader that
$d$-representation-finite algebras are to be thought of as analogues
of hereditary algebras of finite representation type from the
viewpoint of higher Auslander--Reiten theory. Moreover, the bounded
derived category of $\mmod A$, denoted by $\Db(\mmod A)$, has a
$d$-cluster-tilting subcategory
\[
  \mathcal{U}(A):=\add\setP{M[di]\in\Db(\mmod
    A)}{M\in\mathcal{M}\text{ and }i\in\mathbb{Z}}
\]
which is closed under the $d$-th power of the shift functor of
$\Db(\mmod A)$, see \cite[Thm. 1.21]{iyama_cluster_2011}. Hence, it
follows from \cite[Thm. 1]{geiss_n-angulated_2013} that
$\mathcal{U}(A)$ is a $(d+2)$-angulated category. The category
$\mathcal{U}(A)$ can be thought of as a ``bounded derived category''
of $d$-abelian category $\mathcal{M}(A)$. Note that this case is
particular in that every complex in $\mathcal{U}(A)$ is isomorphic to
a finite direct sum of stalk complexes, reflecting the ``hereditary
nature'' of $\mathcal{M}(A)$.

In this short note, we give an example of a finite dimensional algebra
$A$ of global dimension 4 such that $\mmod A$ has a unique
2-cluster-tilting subcategory $\mathcal{M}(A)$. However, the only
rigid subcategory of $\Db(\mmod A)$ containing $\mathcal{M}(A)$ and
which is closed under the second power of the shift functor of
$\Db(\mmod A)$ is precisely $\mathcal{U}(A)$ as defined above. In this
case, it turns out that $\mathcal{U}(A)$ is neither a
2-cluster-tilting subcategory of $\Db(\mmod A)$ nor is it a
4-angulated category with suspension induced by the second power of
the shift of $\Db(\mmod A)$. This shows that if there is a ``bounded
derived category'' of the 2-abelian $k$-category $\mathcal{M}(A)$,
then in cannot be realized as a 2-cluster-tilting subcategory of
$\Db(\mmod A)$.

\section*{The example}

Let $Q$ be the quiver $1\to 2\to 3\to 4\to 5$ and $A$ be the quotient
of the path algebra $kQ$ by the ideal generated by all paths of length
2 in $Q$. It is well known that the algebras $kQ$ and $A$ are derived
equivalent. The Auslander--Reiten quiver of $\mmod A$ is the
following:
\begin{center}
  \includegraphics{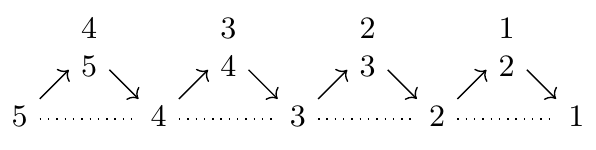}
\end{center}
It is straightforward to verify that
\[
  \mathcal{M}(A):=\add\left(\rep{5}\oplus\rep{4\\5}\oplus\rep{3\\4}\oplus\rep{3}\oplus\rep{2\\3}\oplus\rep{1\\2}\oplus\rep{1}\right)
\]
is the unique $2$-cluster-tilting subcategory of $\mmod A$, see
e.g. \cite[Prop. 6.2]{jasso_n-abelian_2016}. It is also clear that
$\Omega^2(\mathcal{M})\subset\mathcal{M}$, where $\Omega$ is Heller's
syzygy functor.

\begin{proposition}
  Let $\mathcal{U}$ be a rigid subcategory of $\Db(\mmod A)$
  containing $\mathcal{M}(A)$ and satisfying
  $\mathcal{U}[2]=\mathcal{U}$. Then,
  \[
    \mathcal{U}=\mathcal{U}(A):= \add\setP{M[2i]\in\Db(\mmod
      A)}{M\in\mathcal{M}(A)\text{ and }i\in\mathbb{Z}}.
  \]
  Moreover, $\mathcal{U}(A)$ is neither a 2-cluster-tilting
  subcategory of $\Db(\mmod A)$ nor it is a 4-angulated category with
  suspension induced by the second power of the shift of
  $\Db(\mmod A)$.
\end{proposition}

\begin{proof}
  The first claim is a straightforward verification. We provide the
  Auslander--Reiten quiver of $\Db(\mmod A)$ in Figure 1 for the
  convenience of the reader.
  \begin{figure}
	\includegraphics[width=\textheight,angle=90]{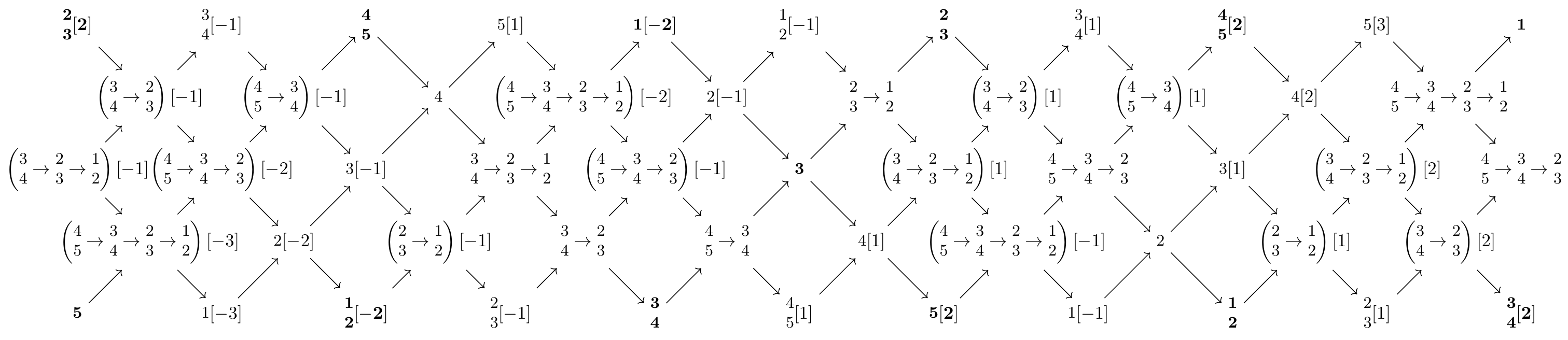}

    \caption{Auslander--Reiten quiver of $\Db(\mmod A).$ The
      subcategory $\mathcal{U}(A)$ is indicated with bold face.}
  \end{figure}
  
  Suppose that $\mathcal{U}(A)$ is a 4-angulated category with respect
  to the second power of the shift functor of $\Db(\mmod A)$ and
  consider the morphism $f\colon\rep{3\\4}\to\rep{2\\3}$. By
  assumption, there exist a 4-angle
  \[
    \begin{tikzcd}
      \rep{3\\4}\rar{f}&\rep{2\\3}\rar{g}&X\rar{h}&Y\rar{i}&\rep{3\\4}[2].
    \end{tikzcd}
  \]
  Since $f$ is not a retraction, \cite[Prop. 2.5
  (a)]{geiss_n-angulated_2013} implies that the morphism $g$ is
  non-zero. Taking a minimal version of $g$, we may assume that
  $X=\rep{1\\2}$. Repeating the same argument for $g$, we deduce that
  $h$ is non-zero and $Y=\rep{1}$. But there are no non-zero
  homomorphisms $\rep{1}\to\rep{2\\3}[2]$ since $\rep{2\\3}$ is an
  injective $A$-module. Hence $i=0$, but this implies that $f$ is a
  section, a contradiction. This shows that $\mathcal{U}(A)$ cannot be
  endowed with the structure of a 4-angulated category for which the
  suspension is given by the second power of the shift functor of
  $\Db(\mmod A)$.  In particular, it cannot be a 2-cluster-tilting
  subcategory of $\Db(\mmod A)$ as this would contradict
  \cite[Thm. 1]{geiss_n-angulated_2013}.
\end{proof}

\begin{remark}
  Note that one can also see that $\mathcal{U}(A)$ is not a
  2-cluster-tilting subcategory of $\Db(\mmod A)$ by observing that
  \[
    \Ext_A^1\left(\rep{3\\4}\to\rep{2\\3}\to\rep{1\\2},\,\rep{3}\right)\neq0
  \]
  while
  \[
    \Ext_A^1\left(\mathcal{U}(A),\,\rep{3\\4}\to\rep{2\\3}\to\rep{1\\2}\right)=0.
  \]
\end{remark}

\bibliographystyle{alpha} \bibliography{zotero}

\end{document}